\documentclass[12pt]{amsart}

\usepackage{amssymb,amsmath,amsthm, url}
\usepackage{comment}
\usepackage{enumitem}
\usepackage[all]{xy}

\setenumerate{label=(\alph*), font=\normalfont}

\newtheorem{thm}[equation]{Theorem}
 \newtheorem{prop}[equation]{Proposition}
 \newtheorem{lem}[equation]{Lemma}
 \newtheorem{cor}[equation]{Corollary}

 \theoremstyle{definition}

 \newtheorem{remark}[equation]{Remark}

\numberwithin{equation}{section}

\newcommand{\bbP}{{\mathbb{P}}}
\newcommand{\bbA}{{\mathbb{A}}}

\newcommand{\bbQ}{{\mathbb{Q}}}

\newcommand{\GL}{\mathrm{GL}}

\newcommand{\Gr}{\operatorname{Gr}}

\newcommand{\Char}{\mathop{\mathrm{char}}\nolimits}

\newcommand{\id}{\mathrm{id}}

\newcommand{\Mor}{\operatorname{Mor}}
\newcommand{\Spec}{\operatorname{Spec}}

\newcommand{\Aut}{\operatorname{Aut}}

\newcommand{\Sym}{\operatorname{S}}

\newcommand{\moduli}[2]{\overline{M}_{#1,#2}}

\DeclareFontEncoding{OT2}{}{} 
\DeclareTextFontCommand{\textcyr}{\fontencoding{OT2}
    \fontfamily{wncyr}\fontseries{m}\fontshape{n}\selectfont}

\author{Mathieu Florence}
\address{Institut de Math\'ematiques de Jussieu, Universit\'e Paris 6,
75005 Paris, France}
\email{mathieu.florence@imj-prg.fr}
\thanks{Florence was partially supported by the French National Agency (Project GeoLie ANR-15-CE40-0012).}

\author{Norbert Hoffmann}
\address{Department of Mathematics and Computer Studies,
Mary Immaculate College,
Limerick, Ireland}
\email{norbert.hoffmann@mic.ul.ie}

\author{Zinovy Reichstein}
\address{Department of Mathematics\\University of British Columbia\\ BC, Canada V6T 1Z2}
\email{reichst@math.ubc.ca}
\thanks{Reichstein was partially supported by
National Sciences and Engineering Research Council of
Canada grant No. 253424-2017} 

\begin{document}

\title[Forms of moduli spaces]
{On the rationality problem for forms of moduli spaces of stable marked curves of positive genus} 


\begin{abstract}
Let $M_{g, n}$ (respectively, $\moduli{g}{n}$)
be the moduli space of smooth (respectively stable) curves of genus $g$
with $n$ marked points. Over the field of complex numbers, it is a classical problem
in algebraic geometry to determine whether or not
$M_{g, n}$ (or equivalently, $\moduli{g}{n}$) is
a rational variety. Theorems of J.~Harris,
D.~Mumford, D.~Eisenbud and G.~Farkas assert that
$M_{g, n}$ is not even unirational for any $n \geqslant 0$ if $g \geqslant 22$.
Moreover, P.~Belorousski and A.~Logan showed that $M_{g, n}$ is unirational
for only finitely many pairs $(g, n)$ with $g \geqslant 1$. Finding 
the precise range of pairs $(g, n)$, where $M_{g, n}$ is rational, stably 
rational or unirational, is a problem of ongoing interest.

In this paper we address the rationality problem for twisted forms of $\moduli{g}{n}$
defined over an arbitrary field $F$ of characteristic 
$\neq 2$. We show that all $F$-forms of $\moduli{g}{n}$ are stably rational 
for $g = 1$ and $3 \leqslant n \leqslant 4$, $g = 2$ and $2 \leqslant n \leqslant 3$, $g = 3$ and 
$1 \leqslant n \leqslant 14$, $g = 4$ and $1 \leqslant n \leqslant 9$,
$g = 5$ and $1 \leqslant n \leqslant 12$.
\end{abstract}

\maketitle

\section{Introduction}
Let $M_{g, n}$ (respectively $\moduli{g}{n}$)
be the moduli space of smooth (respectively stable) curves of genus $g$
with $n$ marked points. Recall that these moduli spaces are defined
over the prime field ($\bbQ$ in characteristic zero and $\mathbb F_p$ in
characteristic $p$). The purpose of this paper is to address the
rationality problem for twisted forms of $\moduli{g}{n}$. Recall that 
a {\em form} of a scheme $X$ defined over a field $F$ is another 
scheme $Y$, also defined over $F$,  such that $X$ and $Y$ become 
isomorphic over the separable closure $F^{sep}$. We will use 
the terms ``form", ``twisted form" and ``$F$-form"
interchangeably.
Forms of $\moduli{g}{n}$ are of interest because they shed light
on the arithmetic geometry of $\moduli{g}{n}$, and because they
are coarse moduli spaces for natural moduli problems
in their own right; see~\cite[Remark 2.4]{fr}.

This paper is a sequel to~\cite{fr}, where two of us considered twisted forms of $\moduli{0}{n}$.
The main results of~\cite{fr} can be summarized as follows.

\begin{thm} \label{thm.g=0} Let $F$ be a field of characteristic $\neq 2$ and
$n \geqslant 5$ be an integer. Then

\smallskip
(a)  all $F$-forms of $\moduli{0}{n}$ are unirational.

\smallskip
(b) If $n$ is odd, all $F$-forms of $\moduli{0}{n}$ are rational.

\smallskip
(c)
If $n$ is even, then there exist fields $E/F$ and $E$-forms 
of $\moduli{0}{n}$ that are not stably rational (or even retract rational) over $E$.
\end{thm}

In the present paper we will study the rationality problem
for forms of $\moduli{g}{n}$ in the case, where $g \geqslant 1$.
Here the rationality problem 
for the usual (split) moduli space $\moduli{g}{n}$ (or equivalently, for $M_{g, n}$)
over the field of complex numbers is already highly non-trivial. 
Theorems of J.~Harris,
D.~Mumford, D.~Eisenbud~\cite{harris-mumford, eisenbud-harris}
and G.~Farkas~\cite{farkas} assert that
if $g \geqslant 22$, then $M_{g, 0}$ is not unirational (and hence, 
neither is $M_{g, n}$ for any $n \geqslant 0$).
Moreover, work of P.~Belorousski~\cite{belorousski} (for $g = 1$) 
and A.~Logan~\cite{logan} (for $g \geqslant 2$) tells us that $M_{g, n}$ is unirational
for only finitely many pairs $(g, n)$ with $g \geqslant 1$. 
Finding the precise range of pairs $(g, n)$, where 
$M_{g, n}$ is rational, stably rational or unirational, 
is a problem of ongoing interest. In particular,
over $\mathbb C$, $M_{g, n}$ is known to be rational for $1 \leqslant n \leqslant r_g$ and
not unirational for $n \geqslant n_g$, where

\medskip
\begin{center}
 \begin{tabular}{c | c c c c c|} 
 $g$ &  \; 1 \; & \; 2 \; & \; 3 \; & \; 4 \; & \; 5 \; \\ 
 \hline
 $r_g$ & 10 & 12 & 14 & 15 & 12 \\ 
 \hline
 $n_g$ & 11 & - & - & 16 & 15 \\
 \hline
\end{tabular} \, ,
\end{center}

\medskip
\noindent
see~\cite{logan} and~\cite{cf}. Surprisingly, we have not been able to find specific values for
$n_2$ and $n_3$ in the literature, even though Logan showed that they exist; see~\cite[Theorem 2.4]{logan}. The main result of the present paper is as follows.

\begin{thm} \label{thm.main} Let $F$ be a field of characteristic $\neq 2$. 
Then every $F$-form of $\moduli{g}{n}$ is stably rational over $F$ if 

\smallskip
$g = 1$ and $3 \leqslant n \leqslant 4$, 

\smallskip
$g=2$ and $2 \leqslant n \leqslant 3$,

\smallskip
$g=3$ and $1 \leqslant n \leqslant 14$,

\smallskip
$g=4$ and $1 \leqslant n \leqslant 9$,

\smallskip
$g= 5$ and $1 \leqslant n \leqslant 12$.
\end{thm}

Several remarks are in order.

\smallskip
(1) Stable rationality of every form of $\moduli{g}{n}$ is a priori 
much stronger than stable rationality of $\moduli{g}{n}$ itself.
For example, $\moduli{1}{1} \simeq \mathbb P^1$ is rational,
but its forms are conic curves which are not unirational in general.

\smallskip
(2) Theorem~\ref{thm.main} also holds for 
$(g, n) = (1,2)$  (respectively, $(2, 1)$), provided 
$\Char(F) = 0$ (respectively, $\Char(F) \neq 2, 3$);
see Remark~\ref{rem.small}. 

\smallskip
(3) By~\cite[Theorem 6.1(b)]{dr}, every
$F$-form of $\moduli{1}{n}$ is unirational for $3 \leqslant n \leqslant 9$. 

\smallskip
(4) The situation we encountered in Theorem~\ref{thm.g=0}(c), where some 
forms of $\moduli{0}{n}$ are stably rational and others are not, does not arise for any of the pairs $(g, n)$
covered by Theorem~\ref{thm.main}. We do not know if it arises for any pair $(g, n)$ with $g \geqslant 1$ 
and $2g + n \geqslant 5$.

\smallskip
A proof of Theorem~\ref{thm.main} is outlined in Section~\ref{sect.proof} and completed in Sections~\ref{sect.beta} and \ref{sect.alpha}. Our arguments rely on a theorem of B.~Fantechi and A.~Massarenti~\cite{fm} 
describing the automorphism group of $\moduli{g}{n}$; see Section~\ref{sect.massarenti}.

\section{Preliminaries}
\label{sect.prel}

All algebraic groups in this paper will be assumed to be affine, and all algebraic varieties to be quasi-projective.

\subsection{Twisting}
\label{sect.twisting}

Let $G$ be an algebraic group defined over a field $F$, $X$ be 
an $F$-variety endowed with a (left) $G$-action, and
$P \to \Spec(F)$ be a (right) $G$-torsor. The twisted variety  $^P X$ is defined 
as $^P X := (P \times X)/G$, where $G$ acts on $P \times X$ by $g \colon (p, x) \to (p \cdot g^{-1}, g \cdot x)$.
Here $P \times X$ is, in fact, a $G$-torsor over $(P \times X)/G$; in particular, $P \times X \to (P \times X)/G$ 
is a geometric quotient. 
A $G$-equivariant morphism of $F$-varieties $f \colon X \to Y$ 
gives rise to a $G$-equivariant morphism $\id \times f \colon P \times X \to P \times Y$
which descends to an $F$-morphism $^P f \colon ^P X \to ^P Y$. Similarly a $G$-equivariant rational map 
$f \colon X \dasharrow Y$ 
of $F$-varieties induces a rational map $^P f \colon {\, }^P X \dasharrow {\, }^P Y$.
Some basic properties of the twisting operation are summarized in Lemma~\ref{lem.90} below;
see also~\cite[Section 2]{florence} or~\cite[Section 3]{dr}.

\begin{lem} \label{lem.90}
Let $G$ be an algebraic group defined over a field $F$,  $f \colon X \to Y$ and $f' \colon X' \to Y$ be $G$-equivariant morphisms of $F$-varieties,
and $P \to \Spec(F)$ be a $G$-torsor. 

\smallskip
(a) If $f$ is an open (respectively, closed) immersion, then so is ${\, }^P\! f$.

\smallskip
(b) If $f$ is a dominant morphism (respectively, an isomorphism or a birational isomorphism), 
then so is ${\, }^P \! f$.

\smallskip
(c) If $f$ is a vector bundle of rank $r$, then so is ${\, }^P \! f$. In particular,
${\, }^P X$ is rational over ${\, }^P Y$.

\smallskip
(d) ${\, }^P (X \times_Y X')$ is isomorphic to
${\, }^P X \times_{{\, }^P Y} {\, }^P (X')$ over ${\, }^P Y$. 

\smallskip
(e) Moreover, if $f$ and $f'$ are vector bundles, then ${\, }^P (X \times_Y X')$ and
${\, }^P X \times_{{\, }^P Y} {\, }^P (X')$ are isomorphic as vector bundles over ${\, }^P Y$.

\smallskip
(f) If $f$ is a vector bundle of rank $r$, then
the twisted Grassmannian bundle ${\, }^P  \Gr(m, X) \simeq \Gr(m, {\, }^P X)$ is rational over 
${\, }^P Y$ for any $1 \leqslant m \leqslant r - 1$. In particular, 
${\, }^P \bbP(X) \simeq \bbP({\, }^P X)$ is rational over ${\, }^P Y$.
\end{lem}

Here when we say that $f$ is a vector bundle, we are assuming that $G$ acts on $X$ by vector bundle automorphisms (and similarly for $f'$).
That is, for any $g \in G$ and $y \in Y$, $g$ restricts to a linear map between the fibers $f^{-1}(y)$ and $f^{-1}(g(y))$.

\begin{proof} For a proof of (a) and (b), see~\cite[Corollary 3.4]{dr}. 

\smallskip
(c) The first assertion is a consequence of Hilbert's Theorem 90. 
The second assertion follows from the first, since the vector bundle
${\, }^P f \colon {\, }^P X \to {\, }^P Y$ becomes trivial after passing to
some dense Zariski open subset of ${\, }^P Y$.

\smallskip
(d) The morphism $\phi \colon P \times (X \times_Y X') \to (P \times X) \times_Y (P \times X')$ over $Y$ given by
$(p, x, x') \mapsto \big( (p, x), (p, x') \big)$ descends to
a morphism $\overline{\phi} \colon {\, }^P (X \times_Y X') \to {\, }^P X \times_{{\, }^P Y} {\, }^P (X')$ over ${\, }^P Y$.
Here the unmarked direct products are assumed to be over $\Spec(F)$. To show that $\overline{\phi}$ is an isomorphism, we may pass to a splitting field $F'/F$ for $P$. 
Over $F'$, the $G$-torsor $P \to \Spec(F)$ becomes split, i.e., $P_{F'} \simeq G_{F'}$. Thus over $F'$, the morphism $P \times X \simeq G \times X \to X$ given by 
$(g, x) \mapsto g \cdot x$ is a $G$-torsor. This yields a natural isomorphism between ${\, }^P X$ and $X$. 
Identifying ${\, }^P (X')$ with $X'$, ${\, }^P Y$ with $Y$, and ${\, }^P (X \times_Y X')$ with $X \times_Y X'$ in a similar manner, we see that over $F'$, $\overline{\phi}$ becomes the identity map $X \times_Y X' \to X \times_Y X'$. Hence, $\overline{\phi}$ is an isomorphism over $F$.

\smallskip
(e) To check that $\overline{\phi}$ is an isomorphism of vector bundles over ${\, }^P Y$, we may, once again, pass to a splitting field $F'/F$ for $P$.
In the proof of part (d), we identified ${\, }^P (X')$ with $X'$ and ${\, }^P (X \times_Y X')$ with $X \times_Y X'$ after passing to $F'$.
Now we observe that these identifications are, in fact, isomorphisms of vector bundles over ${\, }^P Y$ (which we identified with $Y$). 
Modulo these identifications, $\overline{\phi}$ is 
the identity map $X \times_Y X' \to X \times_Y X'$, and part (e) follows.

\smallskip
(f) The second assertion follows from the first by setting $m = 1$.

Rationality of $\Gr(m, {\, }^P X)$  over ${\, }^P Y$ follows from the fact that any vector bundle, 
and in particular the vector bundle ${\, }^P f \colon {\, }^P X \to {\, }^P Y$, is locally trivial 
in the Zariski topology. 

To show that ${\, }^P \Gr(m, X)$ is isomorphic to  $\Gr(m, {\, }^P X)$ over ${\, }^P Y$, recall that
$\Gr(m, X)$ is the quotient of the dense open subset $(X^m)_0$ of the $n$-fold fibered product 
$X^m = X \times_Y \dots \times_Y X$ consisting of linearly independent $m$-tuples by the group 
$\GL_m$ (over $Y$). To construct ${\, }^P \Gr(m, X)$, we proceed as follows. First take the
quotient of the product $P \times (X^m)_0$ by the action of 
$\GL_m$. This action is trivial on the first factor, so we
obtain $P \times \Gr(m, X)$. Now take the quotient of $P \times \Gr(m, X)$ by $G$ to arrive at ${\, }^P \Gr(m, X)$.

To construct $\Gr(m, {\, }^P X)$, we also start with $P \times (X^m)_0$ and take the quotients by the same groups, but
in reverse order. First we take the quotient of $P \times (X^m)_0$
by $G$ to obtain ${\, }^P ((X^m)_0) \simeq (({\, }^P X)^m)_0$ (see parts (d) and (e)); the quotient of $(({\, }^P X)^m)_0$
by $\GL_m$ is $\Gr(m, {\, }^P X)$. Since the actions of $\GL_m$ and $G$ on $P \times_F (X^m)_0$ 
commute, we conclude that ${\, }^P \Gr(m, X)$ and $\Gr(m, {\, }^P X)$ are isomorphic over ${\, }^P Y$.
\end{proof}

The $F$-forms of a variety $X$ are in a natural bijective correspondence with $H^1(F, \Aut(X))$.
Here $\Aut(X)$ is a functor which associates to a scheme $S/F$ the abstract group $\Aut(X_S)$. 
In general this functor is not representable by an algebraic group defined over $F$.
If it is, one usually says that $\Aut(X)$ is an algebraic group. In this case the bijective correspondence between $H^1(F, \Aut(X))$ 
(which may be viewed as a set of $\Aut(X)$-torsors $P \to \Spec(F)$) and the set of $F$-forms of $X$ (up to $F$-isomorphism)
can be described explicitly as follows. An $\Aut(X)$-torsor $P \to \Spec(F)$ 
corresponds to the twisted variety $^P X$, and a twisted form $Y$ of $X$
corresponds to the isomorphism scheme $P = \operatorname{Isom}_F(X, Y)$, which is naturally 
an $\Aut(X)$-torsor over $\Spec(F)$; see~\cite[Section III.1.3]{serre-gc}, ~\cite[Section 11.3]{springer}.

\subsection{\'Etale algebras}
\label{sect.etale-algebra}

An \'etale algebra $A/F$ is a commutative $F$-algebra of the form
$F_1 \times \dots \times F_r$, where each $F_i$ is a finite separable 
field extension of $F$. $n$-dimensional \'etale algebras over $F$ are
$F$-forms of the split \'etale algebra $A = F \times \dots \times F$ ($n$ times).
The automorphism group of this split algebra is the symmetric group $\Sym_n$,
permuting  the $n$ factors of $F$. Thus $n$-dimensional \'etale algebras over $F$
are in a natural bijective correspondence with the Galois cohomology set $H^1(F, \Sym_n)$;
see, e.g., Examples 2.1 and 3.2 in~\cite{gms}.

\subsection{Weil restriction}
\label{sect.weil}

Let $A$ be an \'etale algebra over $F$ and $X \to \Spec(A)$ be 
a variety defined over $A$. The Weil restriction (or Weil transfer)
of $X$ to $F$ is, by definition, an $F$-variety $R_{A/F}(X)$ satisfying
\begin{equation}
\label{e.weil}
\Mor_{F} (Y, R_{A/F}(X)) \simeq \Mor_{A}(Y_A, X)  , \end{equation}
where $Y_A := Y \times_{\Spec(F)} \Spec(A)$, $\Mor_F(Y, Z)$ denotes the set of
F-morphisms $Y \to Z$, and $\simeq$ denotes an isomorphism of functors (in $Y$).
For generalities on this notion we refer the reader to~\cite[Section 7.6]{blr}.
For a brief summary, see~\cite[Section 2]{karpenko}.
In particular, it is shown in~\cite[Theorem 4]{blr} that if $X$ is quasi-projective 
over $A$, then $R_{A/F}(X)$ exists. Note that uniqueness of $R_{A/F}(X)$ follows 
from~\eqref{e.weil} by Yoneda's lemma.

The following properties of Weil restriction will be helpful in the sequel.

\begin{lem} \label{lem.weil}  Let $A/F$ be an \'etale algebra and
$X$ be a (quasi-projective) variety defined over $A$. 


\smallskip
(a) Let $V$ be a free $A$-module of finite rank, and $X = \bbA_A(V)$ 
be the associated affine space. Then $R_{A/F}(X) = \bbA_F(V)$, where we
view $V$ as an $F$-vector space.

\smallskip
(b) If $X$ and $Y$ are birationally isomorphic over $A$, then $R_{A/F}(X)$
and $R_{A/F}(Y)$ are birationally isomorphic over $F$.

\smallskip
(c) If $X$ is a rational variety over $A$, then $R_{A/F}(X)$ is
rational over $F$.
\end{lem}

\begin{proof} 
(a) follows directly from~\eqref{e.weil}. For details, see \cite[Lemma 1.2]{karpenko}.

\smallskip
(b) Since $X$ and $Y$ are birationally isomorphic, there exists a variety $U$ defined over $A$ and open immersions
$i \colon U \hookrightarrow X$ and $j \colon U \hookrightarrow Y$. After replacing $U$ by 
an open subvariety, we may assume that $U$ is quasi-projective (we may even assume that $U$ is affine). 
Since Weil restriction commutes with open immersions, $i$ and $j$ induce open immersions of 
$R_{A/F}(U)$ into $R_{A/F}(X)$ and $R_{A/F}(Y)$, respectively, and part (b) follows.

\smallskip
(c) By our assumption, $X$ is birationally isomorphic to $Y = \bbA^d$ over $A$, where $d$ is the dimension of $X$. 
By part (b), $R_{A/F}(X)$ and $R_{A/F}(Y)$ are birationally isomorphic over $F$, and by part (a),
$R_{A/F}(Y)$ is an affine space over $F$. 
\end{proof}

In the special case where $X$ is defined over $F$, the Weil transfer
$R_{A/F}(X_A)$ can be explicitly described as follows. The symmetric 
group $\Sym_n$ acts on the $n$-fold direct product $X^n$ by permuting 
the factors. If $P \to \Spec(F)$ is a $\Sym_n$-torsor, and $A/F$ is 
the \'etale algebra of degree $n$ representing the class of $P$ in 
$H^1(F, \Sym_n)$, then $R_{A/F}(X_A) = {\, }^P (X^n)$; see, e.g.,
\cite[Proposition 3.2]{dr}.

\subsection{Automorphism of marked curves}
We shall need the following well-known result in the sequel; see,
e.g., in~\cite[Corollary IV.4.7]{hartshorne} for $g = 1$ 
and \cite[Exercise V.1.11]{hartshorne} for $g \geqslant 2$.

\begin{prop} \label{prop.stable}
Suppose $2g + n \geqslant 5$. Then $\Aut(C, p_1, \dots, p_n) = \{ 1 \}$
for a general point $(C, p_1, \dots, p_n)$ of $\moduli{g}{n}$
(or equivalently, of $M_{g, n}$). 
\qed
\end{prop}

Note that the inequality $2g + n \geqslant 5$ is satisfied for every pair of integers $(g, n)$ appearing in Theorem~\ref{thm.main}.

\subsection{Automorphisms and forms of $\moduli{g}{n}$}
\label{sect.massarenti}

The following theorem is the starting point of our investigation.

\begin{thm} \label{thm.massarenti} {\rm (A.~Massarenti~\cite{massarenti}, B.~Fantechi and
A.~Massarenti~\cite{fm})}
Let $F$ be a field of characteristic $\neq 2$. If $g, n \geqslant 1$, $(g, n) \neq (2, 1)$
and $2g + n \geqslant 5$, then the natural embedding
$\Sym_n \to \Aut_F(\moduli{g}{n})$ is an isomorphism. \qed
\end{thm}

Using the bijective correspondence between $F$-forms of $X$ and
$\Aut(X)$-torsors $P \to \Spec(F)$ described at the end of 
Section~\ref{sect.twisting}, we obtain the following.

\begin{cor} \label{cor.forms} For $F, g, n$ as in Theorem~\ref{thm.massarenti},
every $F$-form of $\moduli{g}{n}$ is $F$-isomorphic to
${\,}^{P} \moduli{g}{n}$ for some $\Sym_n$-torsor $P \to \Spec(F)$.
\qed
\end{cor}

\begin{remark} \label{rem.small}
Theorem~\ref{thm.main} also holds in the following cases.

\smallskip
(a) $g= 2$ and $n = 1$, and $\Char(F) = 0$,

\smallskip
(b) $g = 1$ and $n = 2$ and $\Char(F) \neq 2$ or $3$.

\smallskip
\noindent
In case (a), $\moduli{2}{1}$ has no non-trivial automorphisms 
by~\cite[Theorem 1]{fm} and hence, no non-split forms. On the other hand,
the split form of $\moduli{2}{1}$ is known to be rational; see~\cite{cf}. 

In case (b), the automorphism group $G$ of $\moduli{1}{2}$ is non-trivial; 
however, it is special; see~\cite[Proposition 2.4]{fm}. In other words,
every $G$-torsor over a field is split.
As a consequence, $\moduli{1}{2}$ has no non-split forms 
(see~\cite[Remark 6.4]{dr}) and the split 
form of $\moduli{1}{2}$ is rational (see~\cite{cf}).
\end{remark}

\begin{remark} \label{rem.lin}
We do not know if $\moduli{g}{n}$ can be replaced by $M_{g, n}$ 
in the statement of Theorem~\ref{thm.massarenti}. 
If so, then $\moduli{g}{n}$ can also be replaced 
by $M_{g, n}$ in the statements of Theorems~\ref{thm.main}. 
The proof remains unchanged.
\end{remark}

\section{Proof of Theorem~\ref{thm.main}: the overall strategy} 
\label{sect.proof}

Let $(C, p_1, \dots, p_n)$ be a point of $M_{g, n}$.

\medskip
{\bf Case I.} We define a vector space $V$ of dimension $d$ as follows.

\smallskip
\begin{itemize}
\item If $g = 3, 4$ or $5$, then $V = H^0(C, \omega_C)^*$, where $\omega_C$ is the canonical line bundle.
Here $d = g$.

\smallskip
\item
If $g = 1$ and $n = 3$ or $4$, then $V = H^0(C, \mathcal{O}_C(p_1 + \ldots + p_n))^*$.
Here $d = n$. 

\smallskip
\item
If $g = 2$ and $n = 2$, then $V = H^0(C, \omega_C(p_1 + p_2))^*$. Here $d = 3$.
\end{itemize}

\smallskip
{\bf Case II.}
\begin{itemize}
\item
If $g = 2$ and $n = 3$, then in place of $V$ we define two vector spaces, $V_1 = H^0(C, \omega_C)^*$ and
$V_2 = H^0(C, \mathcal{O}_C(p_1 + p_2 + p_3))^*$. Here $\dim(V_1) = \dim(V_2) = 2$.
\end{itemize}

\begin{remark} \label{rem.vector-bundle}
In Case I, for $(C, p_1, \dots, p_n)$ in a suitably defined open subset $(M_{g, n})_0$ of $M_{g, n}$,
$V$ is the fiber of a vector bundle $E \to (M_{g, n})_0$ obtained via push-forward from a vector bundle over 
the universal curve. In Case II the same is true of both $V_1$ and $V_2$. After replacing 
$(M_{g, n})_0$ by a dense open subset, we may assume without loss of generality that (i) $(M_{g, n})_0$ is
$\Sym_n$-invariant and (ii) $\Aut(C, p_1, \ldots, p_n) = 1$
for any $(C, p_1, \ldots, p_n) \in (M_{g, n})_0$; see Proposition~\ref{prop.stable}.
\end{remark}

In Case I, set 
\[ \text{$X = \{ (C, p_1, \ldots, p_n, B) \in (M_{g, n})_0 \; | \; B$ is a basis of $V$, up to proportionality$\}$.} \] 
Here two bases $B = (v_1, \dots, v_d)$ and $B' = (v_1', \dots, v_d')$ are called proportional if there exists
a $0 \neq c \in k$ such that $v_i' = c v_i$ for every $i = 1, \dots, d$.

In Case II, the vector space $V$ in the definition of $X$ should be replaced by a pair of $2$-dimensional
vector spaces $V = (V_1, V_2)$ and the basis $B$ by a pair $B = (B_1, B_2)$, where $B_1$ is a basis of $V_1$ and $B_2$ 
is a basis of $V_2$. We identify two such pairs, $B = (B_1, B_2)$ and $B' = (B_1', B_2')$,  
if $B_1$ is proportional to $B_1'$ and $B_2$ is proportional to $B_2'$.

In Case I, choosing a basis in $V$ gives rise to the map $f_B \colon C \to \bbP^{d -1}$. Two bases, $B$ and $B'$,
are proportional if and only if $f_B = f_{B'}$. In Case II,
we obtain two maps, $f_{B_1}, f_{B_2} \colon C \to \bbP^1$, which can be combined into a single morphism
$f_B = f_{B_1} \times f_{B_2} \colon C \to \bbP^1 \times \bbP^1$. Once again, we
identify $B = (B_1, B_2)$ and $B' = (B_1', B_2')$ if and only if $f_B = f_{B'}$.

In each case we will consider a diagram of the form
 \begin{equation} \label{e.diagram-split} 
 \xymatrix{    
               &                             & X \ar@{->}[dl]_{\alpha} \ar@{-->}[dr]^{\beta}              &               &  \\
M_{g, n}   &  (M_{g, n})_0 \ar@{_{(}->}[l]^{\text{open}}  &                                              & Y.   &  } 
\end{equation}
Here $\alpha$ is the natural projection $(C, p_1, \ldots, p_n, B) \to (C, p_1, \ldots, p_n)$ which ``forgets" the basis $B$. 
The second projection $\beta$ ``forgets" the curve $C$, and maps $(C, p_1, \dots, p_n, B)$ to a suitable configuration space $Y$
for the remaining data. Specifically, we define $Y$ and $\beta$ as follows.

\medskip
{\bf Case I.}

\smallskip
\begin{itemize}
\item If $g = 3, 4$ or $5$, then $f_B \colon C \to \bbP^{g-1}$ is the canonical embedding. We define $Y = (\bbP^{g-1})^n$, and 
$\beta(C, p_1, \ldots, p_n, B) = (f_B(p_1), \ldots, f_B(p_n)) \in Y$.

\smallskip
\item
Let $g = 1$ and $n = 3$ or $4$. The constant function $1 \in H^0(C, \mathcal{O}_C(p_1 + \ldots + p_n))$ cuts out 
a hyperplane $L \subset \bbP^{n-1}$ passing through $f_B(p_1), \ldots, f_B(p_n)$. We set $Y \subset ((\bbP^{n-1})^*)^n$ 
to be the locally closed subvariety consisting of $n$-tuples $p_1, \dots, p_n$ 
such that $p_1, \ldots, p_n$ are linearly dependent (i.e., lie in a hyperplane) in $\bbP^{n-1}$ but any $n-1$ of them
are linearly independent and define
$\beta(C, p_1, \dots, p_n) = (f_B(p_1), \ldots, f_B(p_n))$. 

\smallskip
\item
If $g = 2$ and $n = 2$, then for $p_1, p_2$ in general position on $C$, the image of $f_B$ in $\bbP^2$ is a quartic curve $C'$
with a node at $p = f_B(p_1) = f_B(p_2)$, and $f_B \colon C \to C'$ is the normalization map; see~\cite[Example 5.15]{harris}. 
Moreover, $C'$ has two tangent lines at $p$, $L_1$ and $L_2$, which correspond to $p_1$ and $p_2$ under $f_B$. We thus define
$Y$ as the open subvariety of $(\bbP^2)^* \times (\bbP^2)^*$ parametrizing pairs of distinct lines
and $\beta(C, p_1, p_2, B) = (L_1, L_2)$.
\end{itemize}

\medskip
Case II.

\begin{itemize}
\smallskip
\item
Here $g = 2$ and $n = 3$, and the maps $f_{B_1}$ and $f_{B_2} \colon C \to \bbP^1$ are of degree
$2$ and $3$, respectively; see \cite[Examples 5.11 and 5.13]{harris}.
For $p_1, p_2, p_3$ in general position, $C':= f_B(C)$
is a curve of bidegree $(3, 2)$ in $\bbP^1 \times \bbP^1$,
$f_B = f_{B_1} \times f_{B_2}$ is an isomorphism between $C$ and $C'$, 
and $f_{B_2}(p_1) = f_{B_2}(p_2) = f_{B_2}(p_3)$ in $\bbP^1$. We set
$Y$ to be the open subvariety of $(\bbP^1)^3 \times \bbP^1$ consisting of elements of the form
$\big( (a_1, a_2, a_3), b \big)$, where $a_1, a_2, a_3 \in \bbP^1$ are distinct  
and $\beta(C, p_1, p_2, p_3, B) = \big( (f_{B_1}(p_1), f_{B_1}(p_2), f_{B_1}(p_3)), f_{B_2}(p_1) \big)$.
\end{itemize}

\smallskip
By Corollary~\ref{cor.forms} it suffices to show that the twisted variety $^P \, M_{g, n}$ is stably rational over $F$ for every
$\Sym_n$-torsor $P \to \Spec(F)$. Here $1 \leqslant g \leqslant 5$, and $(g, n)$ is one of the pairs appearing in Theorem~\ref{thm.main}.
Twisting the diagram~\eqref{e.diagram-split} by $P$ and applying Lemma~\ref{lem.90}, we obtain the following diagram of twisted varieties.
\begin{equation} \label{e.diagram-twisted} 
 \xymatrix{    
                &                   & { }^P  X \ar@{->}[dl]_{{\, }^P \!  {\Large \alpha}} \ar@{-->}[dr]^{{\,}^P \! {\Large \beta}} &  & \\
  {\, }^P M_{g, n} &  {\, }^P ((M_{g, n})_0)  \ar@{_{(}->}[l]^{\text{open}} &                        & {\, }^P Y. & }   
\end{equation}
In order to complete the proof of Theorem~\ref{thm.main}, we need to establish the following facts for each pair $(g, n)$ 
in Theorem~\ref{thm.main}.

\begin{lem} \label{lem.dominant} The rational map $\beta \colon X \dasharrow Y$ is dominant.
\end{lem}

\begin{lem} \label{lem.alpha} (a) ${\, }^P Y$ is rational over $F$, 

\smallskip
(b) ${\, }^P X$ is rational over ${\, }^P M_{g, n}$,

\smallskip
(c) ${\, }^P X$ is rational over ${\, }^P Y$.
\end{lem}

\section{Proof of Lemma~\ref{lem.dominant}}
\label{sect.beta}

For $g = 3, 4, 5$ we need to show that there is a canonical curve passing through $n$ points $r_1, \dots, r_n$ in general position in $\bbP^{g-1}$.

$g = 3$. Canonical curves of genus $3$ are precisely the smooth quartic curves in $\bbP^2$; 
see~\cite[Example IV.5.2.1]{hartshorne}. Since $\dim \, H^0(\bbP^2, \mathcal{O}(4)) = 15$,
there is a smooth quartic curve passing through $n$ points in $\bbP^2$ in general position 
for any $n \leqslant 14$.

$g = 4$. We will use the fact that a complete intersection of a smooth quadric surface $Q$ and a smooth cubic surface $S$ in $\bbP^3$ is a canonical curve of 
genus $4$; see \cite[Example IV.5.2.2]{hartshorne}.
The dimensions of $H^0(\bbP^3, \mathcal{O}(2))$ and $H^0(\bbP^3, \mathcal{O}(3))$ are $10$ and $20$, respectively.
Hence, as long as $n \leqslant 9$ and $r_1,  \ldots, r_n$ are in general position
in $\bbP^3$, there exist a smooth quadric $Q \subset \bbP^3$ and smooth cubic $S \subset \bbP^3$ 
such that $Q$ and $S$ pass through $r_1, \dots, r_n$ and intersect transversely. The intersection
$Q \cap S$ is then a canonical curve of genus $4$ passing through $r_1, \dots, r_n$.

$g = 5$. Let $n \leqslant 12$. Since $\dim \, H^0(\bbP^4, \mathcal{O}(2)) = 15$, 
for $n$ points $r_1, \dots, r_n$ in general position in $\bbP^4$, 
there exist three linearly independent quadrics, $Q_1$, $Q_2$ and $Q_3$ such that

\smallskip
(i) $Q_1$, $Q_2$ and $Q_3$ pass through $r_1, \dots, r_n$,

\smallskip
(ii) $Q_1$, $Q_2$ and $Q_3$ intersect transversely, and

\smallskip
(iii) $C = Q_1 \cap Q_2 \cap Q_3$ is a smooth curve.

\smallskip
\noindent
By~\cite[Example IV.5.5.3]{hartshorne}, $C$ is a canonical curve of genus $5$. By (i), $C$ passes through $r_1, \dots, r_n$, as desired.

$g = 2$ and $n = 2$. The projection $\beta$ is equivariant with respect to the natural $\GL_3$-action on $X$ and $Y$.
Here $\GL_3$ acts on $(C, p_1, \ldots, p_n, B) \in X$ by linear changes of the basis $B$, leaving $C$ and $p_1, \dots, p_n$
invariant, and on $Y$ via its natural action on $\bbP^2$. Since the $\GL_3$-action on $Y$ is transitive, $\beta$ is dominant.

$(g, n) = (2, 3)$, $(1, 3)$ or $(1, 4)$. Let $G = \GL_2 \times \GL_2$, $\GL_3$ or $\GL_4$, respectively.
In each case, $\beta$ is $G$-equivariant and $G$ acts transitively on $Y$, so the same argument as in the previous case
shows that $\beta$ is dominant.
This completes the proof of Lemma~\ref{lem.dominant}.
\qed

\section{Proof of Lemma~\ref{lem.alpha}}
\label{sect.alpha}

(a) Let $A/F$ be the \'etale algebra associated to the $\Sym_n$-torsor $P \to \Spec(F)$. If $g = 3$, $4$ or $5$, then ${\, }^P Y \simeq R_{A/F}(\bbP^{g-1})$.
If $g = 2$ and $n = 2$, then ${\, }^P Y \simeq R_{A/F}((\bbP^2)^*)$. If $g = 2$ and $n = 3$, then ${\, }^P Y \simeq R_{A/F}(\bbP^{1}) \times \bbP^1$. Here $\simeq$ stands for ``birationally isomorphic over $F$".
In each case ${\, }^P Y$ is rational over $F$ by Lemma~\ref{lem.weil}(c).

If $g = 1$, let $\mathcal{H} \to (\bbP^{n-1})^*$ be the tautological bundle whose fiber over the hyperplane $\{ l = 0 \}$
consists of the points of the affine hyperplane cut out by $l$ in $\bbA^n$. Then $Y$ is $\Sym_n$-equivariantly birationally isomorphic to the $n$-fold fibered product $\bbP(\mathcal{H})^n$ over $(\bbP^{n-1})^*$.
Since $\Sym_n$ acts trivially on $(\bbP^{n-1})^*$, and $(\bbP^{n-1})^*$ is rational over $F$,
it suffices to show that ${\, }^P (\bbP(\mathcal{H})^n)$
is rational over $(\bbP^{n-1})^*$. Choosing projective coordinates $a_1, \dots, a_n$ in
$\bbP^{n-1}$, we can identify the function field $K$ of $(\bbP^{n-1})^*$
with $F(a_i/a_j \, | \, i, j = 1, \ldots, n)$. Over $K$, $\mathcal{H}$ is isomorphic to the $(n-1)$-dimensional
vector subspace on $K^n$ given by $a_1 x_1 + \dots + a_n x_n = 0$ and ${\, }^{P} (\bbP(\mathcal{H})^n)$ is isomorphic to 
$R_{A_K/K} (\bbP(\mathcal{H})_{A_K})$, where $A_K = A \otimes_F K$. 
By Lemma~\ref{lem.weil}(c), $R_{A_K/K} (\bbP(\mathcal{H})_{A_K})$ is rational over $K$. This shows that
${\, }^P (\bbP(\mathcal{H})^n)$ is rational over $(\bbP^{n-1})^*$, as desired.

\smallskip
(b) Case I. Let $E \to (M_{g, n})_0$ be the vector bundle whose fiber over $(C, p_1, \dots, p_n)$ is $V$, as in Remark~\ref{rem.vector-bundle}.
The space of bases in $V$ (up to equivalence) can be identified with a dense
open subset of $\bbP(V^{d})$. Hence, $X$ is $\Sym_n$-equivariantly birationally isomorphic to 
$\bbP(E^d)$ over $(M_{g, n})_0$ and consequently ${\, }^P X$ is birationally isomorphic to ${\, }^P \bbP(E^d)$. On the other hand, by
Lemma~\ref{lem.90}(f), ${\, }^P \bbP(E^d)$ is rational over ${ }^P (M_{g, n})_0$, and by Lemma~\ref{lem.90}(a),  ${ }^P (M_{g, n})_0$ is a
dense open subset of ${ }^P (M_{g, n})$. We conclude that ${\, }^P X \simeq {\, }^P \bbP(E^d)$ 
is rational over ${ }^P M_{g, n}$.

Case II. Now suppose $(g, n) = (2, 3)$. Let $E_1$ and $E_2$ be the rank $2$ vector bundles over $(M_{g, n})_0$ whose fibers 
over $(C, p_1, \dots, p_n)$ are $V_1$ and $V_2$, respectively, as in Remark~\ref{rem.vector-bundle}. 
The space of bases in $V_i$ (up to equivalence) can be identified with a dense
open subset of $\bbP(V_i^{2})$. Hence, $X$ is $\Sym_3$-equivariantly birationally 
isomorphic to a dense open subset of $\bbP(E_1^2) \times_{(M_{g, n})_0} \bbP(E_2^2)$ 
over $(M_{g, n})_0$. By Lemma~\ref{lem.90}(d),  ${\, }^P X$ is birationally isomorphic to 
${\, }^P \bbP(E_1^2) \times_{{\, }^P (M_{g, n})_0} {\, }^P \bbP(E_2^2)$ over 
${ }^P (M_{g, n})_0$. By Lemma~\ref{lem.90}(f), each ${\, }^P \bbP(E_i^2)$ is rational over ${ }^P (M_{g, n})_0$
(or equivalently, over ${ }^P M_{g, n})$. Hence, so is ${\, }^P X \simeq {\, }^P \bbP(E_1^2) \times_{{\, }^P (M_{g, n})_0} {\, }^P \bbP(E_2^2)$.
 
 \smallskip
(c) Recall that in Case I each $(C, p_1, \dots, p_n, B) \in X$ gives rise to a map $f_B \colon C \to \bbP^{d-1}$. 
Here $d = \dim(V)$. In Case II, $f_B$ maps $C$ to $\bbP^1 \times \bbP^1$.  
Now observe that in both cases $f_B$ (and thus $B$, up to equivalence) is uniquely determined by
the image of $(C, p_1, \dots, p_n)$ under $f_B$. Indeed, in each case $f_B$ maps $C$ birationally 
onto its image. If $f_B$ and $f_{B'}$ have the same image, then composing $f_{B'}$ with $f_B^{-1}$, 
we obtain an automorphism of $(C, p_1, \dots, p_n)$. On the other hand, for 
$(C, p_1, \dots, p_n) \in (M_{g, n})_0$, $\Aut(C, p_1, \dots, p_n) = 1$; see Remark~\ref{rem.vector-bundle}.

$g = 3$. Using the above observation we can $\Sym_n$-equivariantly identify $X$ with the space
of tuples $(Q, r_1, \dots, r_n)$, where $Q \subset \bbP^2$ is a smooth curve of degree $4$ and 
$r_1, \dots, r_n$ are $n$ distinct points on $Q$. Let $Y_0 \subset Y = (\bbP^2)^n$ be the dense open $\Sym_n$-invariant 
subvariety consisting of $n$-tuples $(r_1, \ldots, r_n)$ such that (i) $(r_1, \dots, r_n)$ impose $n$ independent conditions 
on quartic curves, (ii) there is a smooth quartic curve passing through $r_1, \dots, r_n$.
Let $W \to Y_0 \subset (\bbP^2)^n$ be the vector bundle whose fiber over $(r_1, \dots, r_n)$ is the space of quartic polynomials 
in 3 variables vanishing at $r_1, \dots, r_n$. Then over $Y_0$, $X$ is $\Sym_n$-equivariantly birationally isomorphic to 
$\bbP(W)$. By Lemma~\ref{lem.90},  ${\, }^P X \simeq {\, }^P \bbP(W)$ is rational over ${\, }^P Y$.

$g = 4$. Recall that $C \subset \bbP^3$ is a canonically embedded curve of genus $4$ if and only if $C$ is a complete intersection of an irreducible quadric surface $Q$ and an irreducible cubic surface $S$ in $\bbP^3$.  Moreover, the quadric $Q$ is uniquely determined by $C$, and the cubic polynomial $s$ which cuts out $S$, is uniquely determined up to replacing $s$ by $s' = c \cdot s + l q$, where $c \in F^*$ is a non-zero constant, $q$ is the quadratic form 
cutting out $Q$, and $l$ is a linear form. Conversely, any irreducible non-singular curve in $\bbP^3$, which is
a complete intersection of an irreducible quadric and an irreducible cubic, is a canonically embedded curve 
of genus $4$; see \cite[Example IV.5.2.2]{hartshorne}. 

Let $Y_0 \subset Y = (\bbP^3)^n$ be the open subset consisting of $n$-tuples of points imposing independent conditions of quadrics and cubics in $\bbP^3$. Let $W$ be the space of $(n+1)$-tuples $(q, r_1, \dots, r_n)$, where $(r_1, \dots, r_n) \in Y_0$ 
and $q \in H^0(\bbP^3, \mathcal{O}(2))$ vanishes at $r_1, \dots, r_n$. The natural projection $W \to Y_0$ given by 
$(q, r_1, \ldots, r_n) \mapsto (r_1, \ldots, r_n)$ is a vector bundle of rank $10 - n$.
The fiber of the projective bundle $\bbP(W)$ over $(r_1, \ldots, r_n) \in Y_0$ parametrizes quadric surfaces $Q \subset \bbP^3$ passing through 
$r_1, \ldots, r_n$. Now let $W'$ be the vector bundle of rank $20-n$ over $\bbP(W)$,
whose fiber over $(Q, r_1, \dots, r_n)$ consists of cubic forms $s \in H^0(\bbP^3, \mathcal{O}(3))$
vanishing at $r_1, \dots, r_n$. Let $W'' \subset W$ be the subbundle, whose fiber over $(Q, r_1, \dots, r_n)$ consists of cubic forms $l \cdot q$, where $q \in H^0(\bbP^3, \mathcal{O}(2))$ cuts out $Q$ and $l$ ranges over $H^0(\bbP^3, \mathcal{O}(1))$. 
Now set $\overline{W} = W'/W''$. A general point $(S, Q, r_1, \ldots, r_n)$ of $\bbP(\overline{W})$ gives rise to a canonical curve $Q \cap S \subset
\bbP^3$ of genus $4$ passing through $r_1, \dots, r_n$.
Thus $X$ is $\Sym_n$-birationally isomorphic to $\bbP(\overline{W})$ over $Y_0$, and we obtain the following diagram of $\Sym_n$-equivariant maps
\[ \xymatrix{X \ar@{-->}[r]^{\sim \quad}  & \bbP(\overline{W}) \ar@{->}[d] \\
             & \bbP(W) \ar@{->}[d]   \\
        Y \ar@{-->}[r]^{\sim \quad} & Y_0 .} \]
Twisting by the $\Sym_n$-torsor $P \to \Spec(F)$, we obtain a diagram
\[ \xymatrix{ {\,}^P X \ar@{-->}[r]^{\sim \quad}  & {\, }^P \bbP(\overline{W}) \ar@{->}[d] \\
             & {\, }^P \bbP(W) \ar@{->}[d]   \\
 {\, }^P Y  \ar@{-->}[r]^{\sim \quad}            & {\, }^P Y_0 .} \]
By Lemma~\ref{lem.90}, ${\, }^P X$ is rational over ${\, }^P \bbP(W)$ and
${\, }^P \bbP(W)$ is rational over over $ {\, }^P Y$.

$g = 5$. Recall that a general canonical curve $C' = f_B(C)$ of genus $5$
is a complete intersection of three quadric hypersurfaces $Q_1$, $Q_2$ and $Q_3$
in $\bbP^4$. Let $q_i \in H^0(\bbP^4, \mathcal{O}(2))$ be a defining equation for $Q_i$.
Then the span of $q_1$, $q_2$ and $q_3$ is uniquely determined by the canonical curve $C'$, because 
$H^0(\bbP^4, \mathcal{I}_{C'}(2))$ is $3$-dimensional. Conversely, a $3$-dimensional subspace of $H^0(\bbP^4, \mathcal{O}(2))$
in general position cuts out a canonical curve of genus $5$ in $\bbP^4$; see~\cite[Example IV.5.5.3]{hartshorne}.

Let $Y_0 \subset Y = (\bbP^4)^n$ be the open subset consisting of $n$-tuples of points imposing independent conditions of quadrics.
Let $W$ be the space of $(n+1)$-tuples $(q, r_1, \dots, r_n)$, where $(r_1, \dots, r_n) \in Y_0$ 
and $q \in H^0(\bbP^4, \mathcal{O}(2))$ vanishes at $r_1, \dots, r_n$. The natural projection 
$W \to Y_0 \subset (\bbP^4)^n$ given by 
$(q, r_1, \ldots, r_n) \mapsto (r_1, \ldots, r_n)$ is a vector bundle of rank $15 - n$.
By the above description, $X$ is $\Sym_n$-equivariantly birationally isomorphic to the total 
space of the Grassmannian bundle $\Gr(3, W)$. Twisting by $P$, we obtain the following diagram
\[ \xymatrix{ {\, }^P X \ar@{-->}[r]^{\sim \quad \quad}  &  {\, }^P \Gr(3, W) \ar@{->}[d]   \\
 {\, }^P Y  \ar@{<--}[r]^{ \sim}    &  {\, }^P Y_0 .} \]
 By Lemma~\ref{lem.90}, we conclude that ${\, }^P X$ is rational over ${\, }^P Y$.

$g = 2$ and $n = 2$. Here $Y \subset (\bbP^2)^* \times (\bbP^2)^*$ parametrizes pairs $(L_1, L_2)$ of distinct lines in $\bbP^2$.
Let $W \to  Y$ be the vector bundle whose fiber over $(L_1, L_2)$ consists of quartic forms $q \in H^0(\bbP^2, \mathcal{O}(4))$ such that
both $q_{\, | L_1}$ and $q_{ \, | L_2}$ vanish to second order at $p = L_1 \cap L_2$. Then $X$ is $\Sym_2$-equivariantly birationally isomorphic to $\bbP(W)$ over $Y$. By Lemma~\ref{lem.90},
 we conclude that ${\, }^P X$ is rational over ${\, }^P Y$.
 
 $g = 2$ and $n = 3$. Here $X$ is $\Sym_3$-equivariantly birationally isomorphic to $\bbP(W)$, where $W$ is the vector bundle over a suitable dense open subset 
 of $Y \subset (\bbP^1)^3 \times \bbP^1$ whose fiber over $((a_1, a_2, a_3), b)$ consists of bihomogeneous polynomials 
 $\phi \in H^0(\bbP^1 \times \bbP^1, \mathcal{O}(3, 2))$ vanishing at $(a_1, b)$, $(a_2, b)$ and $(a_3, b)$. 
 By Lemma~\ref{lem.90}, ${\, }^P X$ is rational over ${\, }^P Y$.
 
 $g = 1$ and $n = 3$. Here $(C', r_1, r_2, r_3) = f_B(C, p_1, p_2, p_3)$ is a smooth plane cubic curve with three distinct collinear points for every
 $(C, p_1, p_2, p_3) \in M_{1, 3}$. Conversely, every smooth cubic curve $C' \subset \bbP^2$ with three distinct
 collinear points $r_1, r_2, r_3 \in C'$ is of the form $f_B(C, p_1, p_2, p_3)$ for some $(C, p_1, p_2, p_3, B) \in X$,
 because $\mathcal{O}_{C'}(r_1 + r_2 + r_3) = \mathcal{O}_{C'}(1)$. Thus $X$ is $\Sym_3$-equivariantly birationally 
 isomorphic to $\bbP(W)$ over $Y$, where
 $W \to Y$ is the vector bundle whose fiber over $(r_1, r_2, r_3)$ consists of cubic forms in
 $H^0(\bbP^2, \mathcal{O}(3))$ vanishing at $(r_1, r_2, r_3)$.  Twisting by the $\Sym_3$-torsor $P$, we conclude that
 ${\, }^P X \simeq \bbP({\, }^P W)$ is rational over ${\, }^P Y$ by Lemma~\ref{lem.90}(f).
 
 $g = 1$ and $n = 4$. For $(C, p_1, \dots, p_4)$ in general position, $(C', r_1, \ldots , r_4) = 
 f_B(C, p_1, \ldots , p_4)$ is a smooth curve of genus $1$
 in $\bbP^3$ with four coplanar points no three of which are collinear. 
By~\cite[Exercise IV.3.6(b)]{hartshorne}, the space 
$H^0(\bbP^3, \mathcal{I}_{C'}(2))$ of global sections of the ideal sheaf
$\mathcal{I}_{C'}(2)$ is $2$-dimensional. Moreover, if $q_1, q_2$ is a basis of this space,
then $C'$ is a complete intersection of the quadrics $Q_1$ and $Q_2$ cut out by $q_1$ and $q_2$.
Conversely, a complete intersection of two smooth quadrics in $\bbP^3$ in general position is a smooth curve of genus $1$;
see~\cite[Exercise I.7.2]{hartshorne}. We conclude that $X$ is $\Sym_4$-equivariantly birationally isomorphic to the total space 
of the Grassmannian bundle $\Gr(2, W)$ over a suitably defined dense open subvariety $Y_0 \subset Y$, 
where $W \to Y_0$ is the vector bundle whose fiber over $(r_1, \ldots, r_4) \in Y$ 
consists of $q \in H^0(\bbP^3, \mathcal{O}(2))$ vanishing at $(r_1, \ldots, r_4)$. Twisting by a $\Sym_4$-torsor 
$P \to \Spec(F)$, we see that ${\, }^P X$ is birational to $\Gr(2, {\, }^P W)$ over ${\, }^P Y$. 
By Lemma~\ref{lem.90}(f), ${\, }^P X$ is rational over ${\, }^P Y$.

This concludes the proof of Lemma~\ref{lem.alpha} and thus of Theorem~\ref{thm.main}. \qed

\section*{ Acknowledgements} We are grateful to A. Massarenti, M.~Mella, M.~Talpo and A. Vistoli for helpful discussions.


\end{document}